\documentclass{birkau}

\usepackage{amsmath,amssymb,url}
\usepackage[shortlabels]{enumitem} 

\usepackage{amsthm}
\usepackage{amssymb}
\usepackage{mathabx}
\usepackage{mathtools}
\usepackage{tikz}
\usetikzlibrary{cd}
\tikzstyle{every picture} = [scale=.5]

\numberwithin{equation}{section}

\theoremstyle{plain}
\newtheorem{theorem}{Theorem}[section]
\newtheorem{lemma}[theorem]{Lemma}
\newtheorem{proposition}[theorem]{Proposition}
\newtheorem{corollary}[theorem]{Corollary}
\theoremstyle{definition}
\newtheorem{example}[theorem]{Example} 

\newcommand{\va}[1]{{\ensuremath{\mathsf{#1}}}} 
\newcommand{\cls}[1]{{\ensuremath{\mathcal{#1}}}}
\newcommand{\m}{\mathbf} 
\newcommand{\lra}{\leftrightarrow}
\newcommand{\rd}{{/}}
\newcommand{\rrd}{{/\kern-2pt/}}
\newcommand{\ld}{{\backslash}}
\newcommand{\lld}{{\backslash\kern-2pt\backslash}}
\newcommand{\jn}{\vee}
\newcommand{\mt}{\wedge}
\renewcommand{\ln}{{\sim}}
\newcommand{\ga}{\gamma}

\newcommand{\da}{\mathord{\downarrow}}
\newcommand{\ua}{\mathord{\uparrow}}
\newcommand{\sle}{\sqsubseteq}
\newcommand{\nsle}{\nsqsubseteq}

\renewcommand{\sup}{\bigvee}

\newcommand\pair[1]{\langle#1\rangle}
\newcommand\Z{{\mathbb Z}}
\newcommand\rest{\mathord{\upharpoonright}}
\newcommand\cat{\mathbin{\copyright}}
\newcommand\ut{{\mathrm e}}
\newcommand\Si{\operatorname{Si}}
\newcommand\emb{\hookrightarrow}

\newcommand\ct[1]{\mathbf{#1}}
\newcommand\eq{\approx}
\newcommand\fin{\mathit{fin}}
\newcommand\gen[1]{{\langle#1\rangle}}
\newcommand\comp{\mathrel{\mathrlap{{\sqsubseteq}}{\sqsupseteq}}} 
\newcommand\incomp{\parallel} 
\let\tsc\textsuperscript

\begin{document}


\title[Idempotent residuated lattices]{Structure theorems for idempotent\\residuated lattices}


\author[J. Gil-F\'erez]{Jos\'e Gil-F\'erez}
\address{University of Bern\\
Mathematical Institute\\
Alpeneggstrasse 22\\
3012 Bern, Switzerland}
\email{gilferez@gmail.com}

\author[P. Jipsen]{Peter Jipsen}
\address{Chapman University\\
Faculty of Mathematics\\
Keck Center of Science and Engineering\\
1 University Drive\\
Orange, CA 92866, USA }
\email{jipsen@chapman.edu}

\corrauthor[G. Metcalfe]{George Metcalfe}
\address{University of Bern\\
Mathematical Institute\\
Sidlerstrasse 5\\
3012 Bern, Switzerland}
\email{george.metcalfe@math.unibe.ch}

\thanks{The research of the second and third authors was supported by the Swiss National Science Foundation (SNF) grant 200021$\_$165850.}




\begin{abstract}
In this paper we study structural properties of residuated lattices that are idempotent as monoids. We provide descriptions of the  totally ordered members of this class and obtain counting theorems for the number of finite algebras in various subclasses. We also establish the finite embeddability property for certain varieties generated by classes of residuated lattices that are conservative in the sense that monoid multiplication always yields one of its arguments. We then make use of a more symmetric version of Raftery's characterization theorem for totally ordered commutative idempotent residuated lattices to prove that the variety generated by this class has the amalgamation property. Finally, we address an open problem in the literature by giving an example of a noncommutative variety of idempotent residuated lattices that has the amalgamation property.
\end{abstract}

\maketitle


\section{Introduction}\label{sec:intro}

A \emph{residuated lattice} is an algebraic structure $\m A=\pair{A,\mt,\jn,\cdot,\ld,\rd,\ut}$ of type $\langle 2,2,2,2,2,0 \rangle$ such that $\langle A,\mt,\jn \rangle$ is a lattice, $\langle A,\cdot,\ut \rangle$ is a monoid, and $\ld,\rd$ are left and right residuals, respectively, of $\cdot$ in the underlying lattice order, i.e., $b \le a \ld c \iff  a \cdot b \le c  \iff a \le c \rd b$\, for all $a,b,c\in A$. Such structures provide algebraic semantics for substructural logics, as well as encompassing well-studied classes of algebras such as lattice-ordered groups and lattices of ideals of rings with product and division operators (see, e.g.,~\cite{JT02,BT03,GJKO07,MPT10}).

A residuated lattice $\m{A}$ is called \emph{idempotent} if $a \cdot a=a$ for all $a\in A$. Structural properties of idempotent residuated lattices have been studied quite widely in the literature (see, e.g.,~\cite{Dun70,Raf07,CZG09,CZ09,Ols11,MM12,GR15,CC19}), notably for Brouwerian algebras, where the product coincides with the meet, and odd Sugihara monoids, where the product is commutative and the map  $x \mapsto x \ld \ut$ is an involution.  The monoidal structure of any idempotent residuated lattice $\m{A}$ is a unital band and the relation on $A$ defined by $a\sle b\, :\Longleftrightarrow\, a\cdot b=a$ is a preorder that we call the \emph{monoidal preorder} of $\m{A}$; if the product of $\m{A}$ is also commutative, then $\pair{A,\cdot,\ut}$ is a unital meet-semilattice with order $\sle$ and greatest element $\ut$. When $\m{A}$ is totally ordered --- that is, $\m{A}$ is a \emph{residuated chain} --- the product has the further property that $a\cdot b \in\{a,b\}$ for all $a,b\in A$; we call residuated lattices satisfying this condition \emph{conservative}, noting that semigroups with this property are called \emph{quasitrivial} (see, e.g.,~\cite{CDM19}).

The aim of this paper is to obtain structural descriptions of various classes of idempotent residuated chains, recalling that such classes generate  varieties of \emph{semilinear} residuated lattices, also referred to in the literature as \emph{representable} residuated lattices (see, e.g.,~\cite{Raf07}).  In Section~\ref{sec:commutative}, we make use of a description, first given in~\cite{Stan07}, of finite commutative idempotent residuated chains (Theorem~\ref{thm:representation:CIdRL}) to prove that there are $2^{n-2}$ such algebras of size $n\geq 2$ (Theorem~\ref{thm:number:CIpChains}). We then establish a more symmetric version of Raftery's characterization theorem~\cite{Raf07} for commutative idempotent residuated chains (Theorem~\ref{thm:representation}), obtaining also as a corollary (as in~\cite{Raf07}) that the variety of semilinear commutative idempotent residuated lattices  is locally finite.

In Section~\ref{sec:idempotentchains}, we provide a description of finite idempotent residuated chains (Theorem~\ref{thm:representationchains}) and show  (Theorem~\ref{thm:numberchains}, proved independently in~\cite{CDM19}) that the number $\ct I(n)$ of such algebras of size $n\geq 2$ satisfies the recurrence formula $\ct I(2) = 1$, $\ct I(3) = 2$, $\ct I(n+2) = 2\ct I(n) + 2\ct I(n+1)$, yielding
\[
\ct I(n) = \frac{\big(1+\sqrt 3\,\big)^n - \big(1-\sqrt 3\,\big)^n}{2\sqrt 3}.
\]
In Section~\ref{sec:conservative}, we prove that if a variety is  generated by a class of conservative residuated lattices defined relative to the variety of residuated lattices by a set of positive universal formulas over the language $\{\jn, \cdot, \ut\}$, then it has the finite embeddability property (Theorem~\ref{thm:FEP:for:CsRL}). In particular, this is the case for the variety of semilinear idempotent residuated lattices, which is shown to be not locally finite.  We then give a description of finite conservative commutative residuated lattices (Theorem~\ref{thm:Catalan}) and prove (Theorem~\ref{thm:numberconservative}) that the number of such algebras with $n\ge 1$ elements is the $(n-1)$th Catalan number
\[
\ct C(n) = \frac 1{n}\binom{2(n-1)}{n-1}.
\]
Finally, in Section~\ref{sec:amalgamation}, we use Theorem~\ref{thm:representation} and a theorem for amalgamation in varieties of residuated lattices from~\cite{MMT14} to prove that the variety of semilinear commutative idempotent residuated lattices has the amalgamation property (Theorem~\ref{thm:semilinearidempotentamalgamation}). We also solve an open problem in the literature (see,~e.g.,~\cite{MMT14}) by giving an example of a noncommutative variety of residuated lattices (generated by a four element residuated chain) that has the amalgamation property (Theorem~\ref{thm:noncommutativeamalgamation}).


\section{Preliminaries}\label{sec:preliminaries}

In this section, we establish some basic properties for idempotent residuated lattices. Recall first that the \emph{cone} of a residuated lattice $\m A$ is the union of its \emph{positive cone} $\ua \ut$ and \emph{negative cone} $\da \ut$, and that $\m A$ is  \emph{conical} if $A = \ua \ut \cup \da \ut$. An element of a conical residuated lattice is said to have \emph{positive sign} if it belongs to the positive cone, and \emph{negative sign} otherwise. From now on, we also denote the product of elements $a,b$ in a monoid by $ab$.

\begin{lemma}[{Cf.~\cite[Lem.~2.1]{Stan07}}]\label{lem:IdRL:properties}
For any idempotent residuated lattice $\m{A}$ and $x,y \in A$:

\begin{enumerate}[font=\upshape, label={(\roman*)}, itemsep=0.5ex]

\item $x\mt y\le xy\le x\jn y$.

\item If $\ut\le xy$, then $xy = x\jn y$.

\item If $xy\le \ut$, then $xy = x\mt y$.

\item If $x,y \in \ua \ut$, then $xy=x\jn y$, and, if $x,y \in \da \ut$, then $xy=x\mt y$.

\item $\langle\da\ut,\mt,\jn,\Rightarrow,\ut\rangle$ is a Brouwerian algebra, where $x\Rightarrow y :=(x\ld y)\mt \ut$.

\end{enumerate}
\end{lemma}

\begin{proof}\
\begin{enumerate}[label={(\roman*)}, itemsep=0.5ex]
\item $x\mt y = (x\mt y)(x\mt y) \le xy \le (x\jn y)(x\jn y) = x\jn y$. 

\item If $\ut\le xy$, then $y \le xyy = xy$ and $x \le xxy = xy$. So $x\jn y \le xy \le x\jn y$.

\item If $xy\le\ut$, then $xy = xyy \le y$ and $xy = xxy \le x$. So $x\mt y \le xy \le x\mt y$.

\item This follows immediately from (ii) and (iii). 

\item Note that in any residuated lattice $xz\le y\iff z\le x\ld y$, and if $z\le \ut$, then $z\le x\ld y \iff z\le (x\ld y)\mt \ut=x\Rightarrow y$. So $x\mt z\le y\iff z\le x\Rightarrow y$ holds in $\da\ut$, and $\langle\da\ut,\mt,\jn,\Rightarrow,\ut\rangle$ is a Brouwerian algebra. \qedhere
\end{enumerate}
\end{proof}

For any idempotent residuated lattice $\m A$, we define the following useful binary relation on $A$:
 \[
 x\sle y\ :\Longleftrightarrow\ xy = x.
 \]
 The next lemma justifies us in calling $\sle$ the \emph{monoidal preorder} of $\m{A}$.

\begin{lemma}\label{lem:char:finite:idempotent:chains}
For any idempotent residuated lattice $\m A$, the relation $\sle$ is a preorder on $A$ with greatest element $\ut$, where if $\m A$ has a least element, this is also the least element of $\sle$. Moreover, for any $x,y\in A$:
\begin{enumerate}[font=\upshape, label={(\roman*)}, itemsep=0.5ex]
\item If $\ut\le x,y$, then $x\le y \iff y\sle x$.
\item If $x, y\le \ut$, then $x\le y \iff x\sle y$.
\end{enumerate}
\end{lemma}
\begin{proof}
The reflexivity and transitivity of $\sle$ follow  from the idempotence and associativity of the product of $\m A$, respectively. So $\sle$ is a preorder. Since $\ut$ is the identity of the product, it is the greatest element of $\sle$, and if $\m A$ has a least element $\bot$, it is an annihilator of the product and   the least element of $\sle$. 

For (i), if $\ut\le x,y \in A$, then, using Lemma~\ref{lem:IdRL:properties}~(iv), 
\[
x\le y \iff y\jn x = y \iff yx = y \iff y\sle x. 
\]
The proof of (ii) is analogous.
\end{proof}

A residuated lattice $\m{A}$  is \emph{commutative} if $ab = ba$ for all $a,b\in A$; in this case, also $a\ld b = b \rd a$ for all $a,b \in A$ and we drop one division operator from the signature, writing $a \to b$ for $a \ld b$. The next result then follows directly from Lemma~\ref{lem:char:finite:idempotent:chains}, and justifies us in calling the monoidal preorder of a commutative idempotent residuated chain its \emph{monoidal order}.

\begin{corollary}\label{cor:monoidalpo}
The monoidal preorder $\sle$ of any commutative idempotent residuated lattice $\m{A}$ is a meet-semilattice order with greatest element $\ut$;  if $\m{A}$ is totally ordered, then $\pair{A,\cdot,\ut}$ is a totally ordered upper-bounded meet-semilattice.
\end{corollary}

Recall from the introduction that a residuated lattice $\m{A}$ is \emph{conservative} if $ab \in\{a,b\}$ for all $a,b\in A$; in particular, a commutative idempotent residuated lattice is conservative if and only if its monoidal preorder is total. Observe also that the monoidal preorder of a conservative idempotent residuated lattice determines its product: that is, if $x\sle y$, then $xy = x$ by definition; otherwise, $xy = y$. The next lemma shows that this is the case in particular for any idempotent residuated chain. 

\begin{lemma}\label{lem:id:chain:implies:sle:determines:multiplication}
Every idempotent residuated chain $\m{A}$ is conservative and its monoidal preorder therefore determines its product.
\end{lemma}

\begin{proof}
Let $x,y\in A$. Then either $\ut\le xy$ or $xy\le \ut$, and hence, by Lemma~\ref{lem:IdRL:properties}, either $xy=x\jn y$ or $xy =x\mt y$. Since $\m A$ is totally ordered, $xy\in\{x,y\}$. 
\end{proof}

The converse does not hold; every conservative residuated lattice is idempotent, but might not be totally ordered. This is the case, however, for the elements belonging to its cone.

\begin{lemma}\label{lem:cones:of:CsRLs:are:chains}
If $\m A$ is a conservative residuated lattice, then $\pair{\da\ut \cup \ua\ut,\le}$  is a chain.
\end{lemma}
\begin{proof}
By Lemma~\ref{lem:IdRL:properties}, $xy=x\mt y$ for any $x,y$ in the negative cone of $\m{A}$, and therefore conservativity implies that $x\mt y=x$ or $x\mt y=y$, so $\pair{\da\ut,\le}$ is a chain. The argument for the positive cone is symmetrical.
\end{proof}

\begin{example}\label{exam:Sugihar}
The variety \va{OSM} of \emph{odd Sugihara monoids} consists of semilinear commutative idempotent residuated lattices satisfying $(x\to\ut)\to\ut \eq x$, and is generated as a quasivariety (proved in~\cite{Dun70}) by the algebra
\[
\m Z = \pair{\Z,\mt,\jn,\cdot,\to,0},
\]
where $\cdot$ is the meet operation of the total order
\[
\dots \prec -3 \prec 3 \prec -2 \prec 2  \prec -1 \prec 1 \prec 0,
\]
and, by calculation,
\[
x\cdot y = 
\begin{cases}
x \mt y & \text{if } x \le -y\\
x \jn y & \text{if } x > -y\\
\end{cases}
\quad\text{and}\quad
x\to y = 
\begin{cases}
(-x)\jn y & \text{if } x \le y\\
(-x)\mt y & \text{if } x>y.\\
\end{cases}
\]
\end{example}

\begin{figure}
\begin{tikzcd}[column sep = tiny, row sep = small]
 & \va{IdRL} & \\
 & \va{CsRL}\ar[u, -] & \va{CIdRL}\ar[ul, -] \\
\va{SemIdRL}\ar[ur, -] & & \va{CCsRL}\ar[ul, -]\ar[u, -] \\
 & \va{SemCIdRL}\ar[ul, -]\ar[ur, -] &  \\
  & \va{OSM}\ar[u, -]  &
\end{tikzcd}
\caption{Varieties of idempotent residuated latices.}
\label{fig:varieties}
 \end{figure}
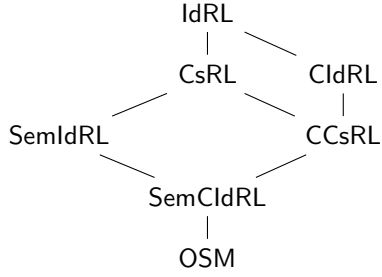

\begin{table}
\begin{center}
\begin{tabular}{|l|l|l|l|}
\hline
Variety 		& LF			& FEP	& AP\\
\hline
\va{IdRL}  		& no\tsc{\cite{MT46}}					& ?		& ?\\
\va{CIdRL}  	& no\tsc{\cite{MT46}}					& yes\tsc{\cite{vaA05}} 		& yes\tsc{\cite{Kam02}} \\
\va{CsRL}  	& no\tsc{Prop.~\ref{prop:SemIdRLnotlocallyfinite}} 				& yes\tsc{Cor.~\ref{cor:conservative:varieties:FEP}}				& ?\\
\va{CCsRL}  	& ? 				& yes\tsc{Cor.~\ref{cor:conservative:varieties:FEP}}					& ?\\
\va{SemIdRL}  	& no\tsc{Prop.~\ref{prop:SemIdRLnotlocallyfinite}}				& yes\tsc{Cor.~\ref{cor:conservative:varieties:FEP}}					& ?\\
\va{SemCIdRL}	& yes\tsc{\cite{Raf07}}	& yes\tsc{\cite{Raf07}}	& yes\tsc{Thm.~\ref{thm:semilinearidempotentamalgamation}}\\
\va{OSM}  	& yes\tsc{\cite{Dun70}} 	& yes\tsc{\cite{Dun70}}	& yes\tsc{\cite{MM12}}\\
\hline
\end{tabular}\\[.1in]
\end{center}
\label{tab:varieties}
\caption{Properties of varieties of idempotent residuated latices.}
 \end{table}

Let us denote the varieties of idempotent and commutative idempotent residuated lattices by \va{IdRL} and \va{CIdRL}, respectively, and the corresponding semilinear varieties generated by their totally ordered members by \va{SemIdRL} and \va{SemCIdRL}. The class of conservative residuated lattices is not closed under products (e.g., the two element residuated chain $\m{2}$ is conservative, but not the product $\m{2}\times\m{2}$) and hence does not form a variety. However, it is defined relative to residuated lattices by a positive universal formula and is therefore a positive universal class. Let us denote the varieties generated by the classes of  conservative and commutative conservative residuated lattices by \va{CsRL} and \va{CCsRL}, respectively, noting that the latter is axiomatized relative to the variety of commutative residuated lattices (via a method described in~\cite{GJKO07}) by the equation (with $s\lra t := (s \to t) \mt (t\to s)$)
\[
\ut \eq ((xy \lra x) \mt \ut) \jn ((xy\lra y)\mt \ut).
\]
Inclusions between the mentioned varieties are displayed in Figure~\ref{fig:varieties}, and Table~\ref{tab:varieties}  summarizes which of these varieties are locally finite (LF), have the finite embeddability property (FEP), or have the amalgamation property (AP), where the superscripts denote where these results were first proved and `$?$' denotes that the problem is still open. Note that amalgamation for $\va{CIdRL}$ is a consequence of Craig interpolation for the corresponding logic (proved in~\cite{Kam02}) and a general theory relating these two properties (see, e.g.,~\cite{GJKO07,MMT14}).


\section{Commutative idempotent residuated chains}\label{sec:commutative}

In this section, we study the structure of commutative idempotent residuated chains. First,  we identify properties of the monoidal order of these algebras and show that in the finite setting they yield a complete structural description (cf.~\cite[Proposition~4.2 and Corollary~4.3]{Stan07}).

Let $\m{C} = \pair{C,\le}$ be any chain. We say that a total order $\sle$ on $C$ with greatest element $\ut$ is \emph{compatible} with $\m{C}$ if
\begin{enumerate}[font=\upshape, label={\arabic*.}, itemsep=0.5ex]
\item whenever $\m{C}$ has a least element $\bot$, also $\pair{C,\sle}$ has least element $\bot$;
\item for all $x,y\in C$, if $\ut\le x,y$, then $x\le y \iff y\sle x$;
\item for all $x,y\in C$, if $x, y\le \ut$, then $x\le y \iff x\sle y$.
\end{enumerate}

\begin{theorem}\label{thm:representation:CIdRL} \
\begin{enumerate}[font=\upshape, label={(\alph*)}, itemsep=0.5ex]
\item
The monoidal order $\sle$ of any commutative idempotent residuated chain $\m{A}$ is total and compatible with $\pair{A,\le}$. 

\item
For any chain $\m{C}=\pair{C,\le}$ and any compatible total order $\sle$ on $\m{C}$,  $\pair{C,\mt,\jn,\cdot,\ut}$ is a commutative idempotent totally ordered monoid, where 
\[
x\cdot y = 
\begin{cases}
x & \text{if } x\sle y,\\
y & \text{otherwise}.
\end{cases}
\]
Moreover, if $\m{C}$ is finite, then $\cdot$ has a (uniquely determined) residual $\to$ and $\pair{C,\mt,\jn,\cdot,\to,\ut}$ is a commutative idempotent residuated chain. 
\end{enumerate}
\end{theorem}

\begin{proof}
Part (a) follows from Lemmas~\ref{lem:char:finite:idempotent:chains} and~\ref{lem:id:chain:implies:sle:determines:multiplication}, and Corollary~\ref{cor:monoidalpo}. For part (b), notice that the defined product is associative, commutative, and idempotent as $x\cdot y$ is the $\sle$-meet of $x$ and $y$ for all $x,y\in C$. The element $\ut$ is the identity, because it is the greatest element of $\sle$. In order to check that $x(y\jn z) = xy\jn xz$ for all $x,y,z\in C$, we may assume without loss of generality that $y\le z$ and prove that $xz = xy\jn xz$, i.e., $xy\le xz$.

We consider the following cases:
\begin{itemize}
\item If $\ut\le x,y,z$, then the product is just the $\le$-join of $x,y,z$, and hence the equation holds. Similarly, if $x,y,z\le \ut$, then the product is the $\le$-meet of $x,y,z$, and the equation holds by distributivity.

\item If $x\le \ut\le y\le z$, then $z\sle y$, by compatibility. If $x\sle z\sle y$, then $xy = x = xz$; if $z\sle x\sle y$, then $xy = x\le z = xz$; if $z\sle y\sle x$, then $xy =  y\le z = xz$.

\item If $x, y\le \ut\le z$, then $xy = x \mt y \le x \le xz$.



\item  If $y\le \ut\le x, z$, then $xy \le x \le x\jn z = xz$.



\item If $y\le z\le \ut\le x$, then $y\sle z$, by compatibility. If $y\sle z\sle x$, then $xy = y\le z =xz$; if $y\sle x\sle z$, then $xy = y\le x = xz$; if $x\sle y\sle z$, then $xy = x = xz$.
\end{itemize}
In the case that $\m C$ is finite, it has a least element $\bot$ and, by compatibility, this is also the least element of $\pair{C, \sle}$. That is, $x \bot = \bot = \bot x$ for all $x \in C$. Since $\m C$ also satisfies $x(y\jn z) = xy\jn xz$ for all $x,y,z\in C$, it follows immediately that the product is residuated.
\end{proof}

As a consequence of this theorem,  counting the commutative idempotent residuated chains of size $n\geq 2$, up to isomorphism, amounts to counting the different compatible total orders $\sle$ on a chain of size $n$.

\begin{theorem}\label{thm:number:CIpChains}
There are $2^{n-2}$ commutative idempotent residuated chains of size $n\geq 2$.
\end{theorem}

\begin{proof}
We determine the number of compatible orderings on a fixed chain $\pair{C,\le}$ of size $n$. The choice of the greatest element $\ut$ of $\sle$ is arbitrary, except that it cannot be the least element. Hence there are $n-1$ choices for $\ut$. For each choice, consider the interval $(\ut,\top] = \{x\in C : \ut < x\le \top\}$ with $k$ elements, and $(\bot,\ut) = \{x\in C : \bot < x < \ut\}$ with $n-2-k$ elements. 

Because of the compatibility conditions, the elements of $(\ut,\top]$ appear in the chain $\pair{C,\sle}$ in the opposite order, while the elements of $(\bot,\ut)$ appear in the same order. As the positions for $\ut$ and $\bot$ are fixed in $\pair{C,\sle}$, all we need to determine is the number of ways of interlayering two sequences of $k$ and $n-2-k$ elements. But this is completely determined by the $k$ places that the elements $(e,\top]$ occupy in the resulting sequence of $n-2$ elements. So there are $\binom{n-2}{k}$ possibilities, and the number of compatible monoidal structures is 
\[
\sum_{k=0}^{n-2}\binom{n-2}{k} = 2^{n-2}.\qedhere
\]
\end{proof}

We now provide a more symmetric version of a representation theorem of Raftery~\cite{Raf07} that describes the structure of all commutative idempotent residuated chains, not just the finite ones. Instead of dividing just  the negative cone of such an algebra into a family of (possibly empty) intervals indexed by the positive elements, as is the case in~\cite{Raf07}, both negative and positive cones are divided into families of nonempty intervals with greatest elements that  together form a {subalgebra}.

Recall first (see, e.g.,~\cite{GJKO07}) that for any residuated lattice $\m A$ and $a\in A$, the map $\ga_a \colon A\to A$ mapping $x$ to $(a\rd x)\ld a$ is a closure operator on $\langle A, \le \rangle$ satisfying $y\cdot\ga_a(x) \le \ga_a(y\cdot x)$. Moreover, when $\m{A}$ is  commutative, the map $\ga_a$ is a nucleus on $\langle A, \le \rangle$ and the algebra $\m A_{\ga_a}=\pair{A_{\ga_a},\mt,\jn_{\ga_a},\cdot_{\ga_a},\to,\ga_a(\ut)}$ with $A_{\ga_a} = \{\ga_a(b) : b \in A\}$, $b \jn_{\ga_a} c = \ga_a(b \jn c)$, and $b \cdot_{\ga_a} c = \ga_a(bc)$ is always a commutative residuated lattice. The next result concerns the particular case when $\m A$ is a commutative idempotent residuated chain and $a = \ut$. For convenience, we define $\ln x = x \to\ut$.

\begin{lemma}\label{lem:Sugihara:skeleton}
If $\m A$ is a commutative idempotent residuated chain, then $\m A_{\ga_\ut}$ is a {subalgebra of $\m A$, that we call its \emph{skeleton}. Moreover, any homomorphism between commutative idempotent residuated chains  restricts to a homomorphism between their skeletons.}
\end{lemma}

\begin{proof}
Clearly, meets and residuals in $\m A_{\gamma_e}$ agree with those in $\m A$. Also  $\ga_\ut(\ut) = \ln\ln \ut = \ut$. Moreover, for all $x,y\in A_{\ga_\ut}$, if $xy = x$, then $x\cdot_{\ga_\ut}y=\ga_\ut(xy) = \ga_\ut(x) = x$, and, similarly, if $xy = y$, then $x\cdot_{\ga_\ut}y = y$. Finally, since $\m A$ is totally ordered, for all $x,y\in A_{\ga_\ut}$, either $x\jn y = x$ or $x\jn y = y$. In the first case, $x\jn_{\ga_\ut} y = \ga_\ut(x\jn y) = \ga_\ut(x) = x$, and in the second case, analogously, $x\jn_{\ga_\ut} y = y$. Hence $\m A_{\ga_\ut}$ is a subalgebra of $\m{A}$.

Finally, let $f\colon\m A\to\m B$ be any homomorphism between commutative idempotent residuated chains. For each $c\in A_{\ga_\ut}$, since $c = \ln\ln c$, also $f(c) = f(\ln\ln c) = \ln\ln f(c) \in B_{\ga_\ut}$.
\end{proof}

\begin{proposition}\label{prop:decomposition:IpCchain}
For any commutative idempotent residuated chain $\m A$:
\begin{enumerate}[font=\upshape, label={(\alph*)}, itemsep=0.5ex]
\item  $\m A_{\ga_\ut}$ is a totally ordered odd Sugihara monoid.
\item For each $c\in A_{\ga_\ut}$, the set $A_c = \{x\in A  : \ga_\ut(x)=c\}$ is an interval of $\m A$ with greatest element $c$.
\item For all $x,y\in A$,
\begin{enumerate}[font=\upshape, label={(\roman*)}, itemsep=0.5ex, topsep=0.5ex] 
\item If $x,y\in A_c$ for some $c\in A_{\ga_\ut}$ with $c\le \ut$, then $xy = x\mt y$.
\item If $x,y\in A_c$ for some $c\in A_{\ga_\ut}$ with $\ut < c$, then $xy = x\jn y$.
\item If  $x\in A_c, y\in A_{d}$ for some $c\neq d\in A_{\ga_\ut}$, then $xy=x \iff cd=c$.
\end{enumerate}
\item For all $x,y\in A$ with $x\in A_c$ for some $c\in A_{\ga_\ut}$,
\[
x\to y = 
\begin{cases}
\ln c \jn y & \text{if } x\le y,\\
\ln c \mt y & \text{if } y < x.
\end{cases}
\]
\end{enumerate}
\end{proposition}
\begin{proof}\
\begin{enumerate}[align = left, leftmargin=0pt, itemindent=\itemindent, labelindent=0pt, labelsep=*, labelwidth=*, label={(\alph*)}, itemsep=0.5ex]
\item $\m A_{\ga_\ut}$ is a subalgebra of $\m A$, and hence also a commutative idempotent residuated chain. But $\ln$ is an involution on $\m A_{\ga_\ut}$ with fixpoint $\ut$, so $\m A_{\ga_\ut}$ is a totally ordered odd Sugihara monoid.

\item Given $c\in A_{\ga_\ut}$ and $x\in A_c$, by definition, $x\le \ga_\ut(x) = c = \ga_\ut(c)$. Also, given $x,y\in A$ such that $x\in A_c$ and $x\le y \le c$, we have $c=\ga_\ut(x)\le \ga_\ut(y) \le \ga_\ut(c) = c$, and hence $\ga_\ut(y) = c$. That is, $y\in A_c$.

\item Part~(i) is immediate, since $x,y\in A_c$ and $c\le \ut$ imply $x,y\le \ut$, and therefore $xy=x\mt y$, by Lemma~\ref{lem:IdRL:properties}.(iv). For (ii), notice that  $\ut\in A_{\ga_\ut}$ and hence $x,y\in A_c$ and $\ut<c$ imply that $\ut<x,y$, and therefore $xy=x\jn y$, by Lemma~\ref{lem:IdRL:properties}.(iv). To prove (iii), we distinguish four cases. Assume that $\ga_\ut(x)=c\neq d=\ga_\ut(y)$ and, without loss of generality, $x<y$. If $x<y\le \ut$, then $c =\ga_\ut(x) <\ga_\ut(y) = d\le \ut$, and hence both $xy=x$ and $cd=c$. If $\ut< x< y$, then $\ut<c=\ga_\ut(x)<\ga_\ut(y)=d$, and both $xy = y$ and $cd = d$. If $x\le \ut < y$ and $xy=x$, then $cd=\ga_\ut(x)\ga_\ut(y) \le\ga_\ut(xy) = \ga_\ut(x) = c\le \ut$, so $cd = c$. Finally, if $x\le \ut < y$ and $xy=y$, notice that $\ut<y\le\ga_\ut(y)$, and therefore $y\ga_\ut(y) = y\jn\ga_\ut(y) = \ga_\ut(y)$. Hence 
\[
\ga_\ut(y) = y\ga_\ut(y) = xy\ga_\ut(y) =  x\ga_\ut(y)\le\ga_\ut(x)\ga_\ut(y)\le \ga_\ut(xy)=\ga_\ut(y),
\]
yielding $cd=d$.

\item We begin by showing that 
\[
x\to x = \begin{cases}x & \text{if } \ut\le x,\\ \ln c & \text{if } x\le \ut,\end{cases}
\]
 and hence, in both cases, $x\to x = \ln c\jn x$. Suppose first that $x\le \ut$. Then $x\to x\le x\to \ut$ and $c \leq \ut$. So $c(\ln c) = c$ and, by part~(c), $x(x\to \ut) = x$, yielding $x\to \ut\le x\to x$. Hence, $x\to x = x\to \ut = \ln c$. Also, $x\le \ut\le \ln c$, so $\ln c\jn x = \ln c$. Suppose next that $\ut\le x$. Then  $x\le a \in A$ implies $xa = a$, so $x\to x = x$. Also, $\ln c \le \ut\le x$, so $\ln c\jn x = x$.

Now suppose that $x\le y$. Then
\[
x(\ln c \jn y) = x(\ln c) \jn xy = x(x\to \ut)\jn  xy \le x \jn y^2 = x \jn y = y. 
\]
Hence $\ln c \jn y \le x\to y$. To establish equality, consider $\ln c \jn y< a \in A$. If $xa =x$, then $a\le x\to x = \ln c\jn x \le \ln c\jn y < a$, a contradiction. So  $xa = a$. That is, $y < a = xa$, and $x\to y <a$.

Suppose finally that $y < x$. If  $y < a \in A$, then $xa = x$ or $xa = a$, and, in both cases, $y< xa$. So  $x\to y \le y$. We show that also $x\to y \le \ln c$. If $x\le \ut$, this is clear because $x\to y \le x\to x = \ln c$. If $\ut < x$ and $a \in A$ satisfies $\ln c < a$ and $xa = a$, then $\ut\le a$ and $x \jn a = xa = a$; that is $y < x\le a$. So $x\to y \le \ln c$. Hence we have shown that $x\to y \le \ln c\mt y$. Let us prove the other inequality. If $x\le \ut$, then $y\le \ut$  and $x(\ln c\mt y) \le xy = x \mt y= y$  and $\ln c\mt y \le x\to y$. Suppose that $\ut\le x$ and hence $x(\ln c) = \ln c \le \ut$. If $y\le \ln c$, then $xy =y$ and $x(\ln c\mt y) \le xy = y$. If $\ln c \le y$, then $x(\ln c\mt y) \le x(\ln c) =\ln c\le y$. Hence, in both cases, $\ln c\mt y \le x\to y$.\qedhere
\end{enumerate}
\end{proof}

Now let $\m S$ be any totally ordered odd Sugihara monoid and let $\mathcal{X} = \{\pair{X_c,\le_c} : c\in S\}$ be a family of (disjoint) chains such that  each $c\in S$ is the greatest element of $X_c$. We define for all $a,b \in S$ with $x\in X_{a}$ and $y\in X_{b}$,
\[
x\preceq y\ \,:\Longleftrightarrow\, \
 a <  b\,\ \text{ or }\,\ (a=b\, \text{ and }\, x\le_{a} y). 
\]
Then $\preceq$ is a total order on
\[
S \otimes \mathcal{X} := \bigcup \{X_c : c \in S\}.
\]
We let $\land$ and $\lor$ be the meet and join operations for $\preceq$ and define the algebra
\[
\m{S} \otimes \mathcal{X} := \pair{S \otimes \mathcal{X},\mt,\jn,\cdot,\to,\ut},
\] 
where for $a,b \in S$ and  $x\in X_{a},y\in X_{b}$, 
\[
x\cdot y = 
\begin{cases}
x \mt y & \text{if } a=b\le\ut\\
x \jn y & \text{if } \ut<a=b\\
x & \text{if } a \neq b \text{ and }ab=a\\
y & \text{if } a \neq b \text{ and }ab=b\\
\end{cases}
\quad\text{ and }\quad
x\to y = 
\begin{cases}
\ln a \jn y & \text{if } x\le y,\\
\ln a \mt y & \text{if } y < x.
\end{cases}
\]

\begin{theorem}[Cf.~\cite{Raf07}]\label{thm:representation}
Let $\m S$ be any totally ordered odd Sugihara monoid and let $\mathcal{X} = \{\pair{X_c,\le_c} : c\in S\}$ be a family of (disjoint) chains such that  each $c\in S$ is the greatest element of $X_c$. Then $\m{S} \otimes \mathcal{X}$ is a commutative idempotent residuated chain satisfying $\m S = (\m{S} \otimes \mathcal{X})_{\ga_\ut}$ and $(S \otimes \mathcal{X})_c = X_c$ for each $c\in S$. Moreover, every commutative idempotent residuated chain has this form.
\end{theorem}

\begin{proof}
The product on $\m{S} \otimes \mathcal{X}$ is clearly commutative and conservative (and therefore idempotent), and extends the product of $\m S$. Hence it suffices to check the following:
\begin{enumerate}[align = left, leftmargin=0pt, itemindent=0pt, labelindent=0pt, labelsep=*, labelwidth=*, label={(\alph*)}, itemsep=0.5ex]
\item \emph{The product is associative.} For $x\in X_{a}$, $y\in X_{b}$, and $z\in X_{c}$, there are four cases:
\begin{itemize}
\item $a=b=c$. Then the product is either meet or join, both of which are associative operations.

\item $a = b\neq c$. Either $ac=bc=a=b$ and $(xy)z = xy = x(yz)$, or $ac=bc=c$ and $(xy)z = z = xz = x(yz)$.

\item $a = c \neq b$. This case is analogous to the previous case.

\item  $a,b,c$ are distinct. In this case, both $(xy)z$ and $x(yz)$ are $x$, $y$, or $z$ depending on whether $(ab)c = a(bc)$ is $a$, $b$, or $c$, respectively.
\end{itemize}

\item \emph{$\ut$ is the identity of the product.} Consider $x\in X_c$. Recall that $\ut$ is the greatest element of the chain $\pair{X_\ut,\le_\ut}$. Hence, if $c=\ut$, then $\ut x = \ut\mt x = x$, and if $c\neq \ut$, then $\ut x = x$ since $\ut c = c$. 

\item \emph{The product is monotone.} Consider $x\in X_{a}$, $y\in X_{b}$, and $z\in X_{c}$ such that $y\preceq z$. First, if $b=c$, then $y\le_{b} z$, and we can distinguish the following cases: if $\ut< a=b$, then $xy = x\jn y \le_{b} x\jn z = xz$; if $a=b\le \ut$, then $xy = x\mt y \le_{b} x\mt z = xz$; if $a\neq b$ and $ab=a$, then $xy = x = xz$; if $a\neq b$ and $ab = b$, then $xy =y\le_{b} z = xz$. In all these cases, $xy\preceq xz$.

On the other hand, if $b < c$, then we have the following cases:
\begin{itemize}
\item $a,b,c$ are distinct. If $ab = a$ and $ac = a$, then $xy = x =xz$. If $ab = a$ and $ac = c$, then it is easy to check that the only possibilities are $a\le b \le \ut\le c$, $b\le \ut\le a\le c$, or $\ut\le b\le a\le c$, and in all three cases $a\le c$. So $xy = x\in X_{a}$ and $xz = z\in X_{c}$, yielding $xy\preceq xz$. Finally, if $ab=b$, then the monotonicity of the product of $\m S$ yields $ac=c$, and hence $xy=y \preceq z = xz$.

\item $a = b$. If $\ut < a$, then $\ut< c$ and so $ac = c$. Hence $xy\in X_{b}$ and $xz = z\in X_{c}$, so $xy\preceq xz$. If $a\le \ut$, then $xy = x\mt y \preceq xz$, by conservativity and the assumption that $y\preceq z$.

\item $a = c$. If $a\le \ut$, then $b < c=a\le \ut$ and so $ab = b$. Hence $xy = y\in X_{b}$ and $xz\in X_{c}$, and therefore $xy\le xz$. If $\ut < a$, then $xy \preceq x\jn z = xz$, by the conservativity and the assumption that $y\preceq z$.
\end{itemize}

\item \emph{$xy\le x \iff y\le x\to x$ for all $x,y\in A$.} Suppose that $x\in X_c$. If $x\le \ut$, then $c\le \ut \le\ln c$ and ${x\to x} = \ln c \jn x = \ln c$; it follows that $c(\ln c) = c$ and $x(x\to x) = x(\ln c) = x$. If  $\ut < x$, then $\ut < c$, and so $\ln c\le \ut\le x$. Hence  $x\to x = \ln c \jn x = x$, so $x(x\to x) = xx = x$. In any case, $x(x\to x) = x$, and, by part~(c), if $y\le x\to x$, then $xy \le x(x\to x) = x$.

Now, if $y\in X_{c'}$ is such that $\ln c \jn x < y$, we can distinguish again the following cases: If $x\le \ut$, then $\ut\le \ln c \le c'$, so $cc' = c'$, and therefore $x < y = xy$. If $\ut < x$, then $\ut< c \le c'$, so $cc' = c'$. So if $c=c'$, then $x < y = x\jn y =xy$, and if $c\neq c'$, then $x < y = xy$. We have proved that $xy\le x$ implies $y\le x\to x$.

\item \emph{The residuation property.} Let $x\in X_c$ be an arbitrary element of $A$ and consider the following cases:
\begin{itemize}
\item $x\le y$. Given that the order is total, either $x(\ln c \jn y) = x(\ln c)$ or $x(\ln c \jn y) = xy$. In both cases, $x(\ln c \jn y)\le y$. Now, if $\ln c \jn y < z$, then $\ln c \jn x < z$, that is, $x\to x < z$, so $xz\neq x$, by part~(d). Hence $y < z = xz$. It follows that $xz\le y \iff z\le x\to y$. 

\item $y < x$. First, if $x\le \ut$, we have $y\le \ut$, and therefore by part~(c), $x(\ln c\mt y) \le xy = x\mt y = y$. On the other hand, if $\ut< x$, then $x(\ln c) = \ln c \le \ut$. If $y\le \ln c$, then $xy =y$ and, as before, $x(\ln c\mt y) \le xy = y$. Finally, if $\ln c \le y$, then $x(\ln c\mt y) \le x(\ln c) =\ln c\le y$. In either case, $x(\ln c\mt y) \le y$. Hence, by part~(c), $z\le x\to y$ implies $xz\le y$. Suppose now that $\ln c\mt y < z$. If $y\le\ln c$, then $y = \ln c\mt y <z$, and by conservativity, $y < xz$. On the other hand, if $\ln c < y$, then $\ln c = \ln c \mt y <z$. Notice that if $x\le \ut$, then $y < x \le \ut\le\ln c < y$, which is impossible. So $\ut<x$ and $\ln c\le \ut$. Hence $xz = z$ implies $x\le z\le c$. By conservativity, $y < xz$. Hence also $xz \le y$ implies $z\le x\to y$.\qedhere  
\end{itemize}
\end{enumerate}
\end{proof}

The preceding theorem can be used (as in~\cite{Raf07}) to prove the following local finiteness result.

\begin{corollary}[\cite{Raf07}]\label{cor:locallyfinite}
The variety of semilinear  commutative idempotent residuated lattices is locally finite.
\end{corollary}
\begin{proof}
We recall first the following well-known criterion for local finiteness (see~\cite{Raf07} for a proof): a variety $\va V$ of finite type is locally finite if there exists a function $f \colon \mathbb{N} \to \mathbb{N}$ such that for all $n\in\mathbb{N}$, every $n$-generated subdirectly irreducible algebra in $\va V$ has at most $f(n)$ elements. 

Consider any commutative idempotent residuated chain $\m A$ and  $Y\subseteq A$. By Theorem~\ref{thm:representation},  $\m A$ decomposes into a totally ordered odd Sugihara monoid $\m S$ and a family of chains $\{\pair{X_c,\leq_c} : c\in S\}$. Let
\[
S' = \{\ga_\ut(a) : a\in Y\} \cup \{\ln \ga_\ut(a) : a\in Y\} \cup \{\ut\}.
\]
 Notice that $S'$ forms a subuniverse of $\m S$. Consider also for each $c \in S'$,  the set $X'_c = (X_c \cap Y) \cup \{c\}$ totally ordered by the restriction $\leq_c'$ of $\leq_c$. It is easy to check that the commutative idempotent residuated chain constructed from $S'$ and the family $\{\pair{X'_c,\leq_c} : c\in S'\}$ is a subalgebra $\m A'$ of $\m A$ containing $Y$. Moreover, if $|Y|=n$, then $|S'| \le 2n+1$  and $|X_c'| \le n$ for each $c\in S'$. Hence $|A'| \le (2n+1)n$. The result now follows directly using the criterion.
\end{proof}


\section{Idempotent residuated chains}\label{sec:idempotentchains}

We turn our attention now to the more general case of idempotent residuated chains. Since the monoidal preorder of such an algebra $\m A$ is not a partial order in general, we define for $x,y \in A$,
\[
x\comp y \iff x\sle y \text{ and }y\sle x \quad\text{ and }\quad
x\incomp y  \iff  x\nsle y \text{ and }y\nsle x.
\]
We also say that $x\in A$ is \emph{central} if it commutes with every other element of $\m A$, i.e., $xy = yx$ for all $y\in A$. The following lemma describes the properties of elements of $\m{A}$ that are central and non-central (cf.~\cite[Proposition~3.1]{CZ09}).

\begin{lemma}\label{lem:IdRC:cases}
For any idempotent residuated chain $\m A$, if two elements do not commute, then they have different signs. Moreover, for each $x\in A$, there are three distinct possibilities:
\begin{enumerate}[font=\upshape, label={(\arabic*)}, itemsep=0.5ex]
\item $x$ is central and for all $y\in A$, either $x\sle y$ or $y\sle x$, and $x\comp y \iff x=y$.
\item $x$ is not central, there is a unique $y\in A$ such that $x$ and $y$ do not commute, and  $x\comp y$.
\item $x$ is not central, there is a unique $y\in A$ such that $x$ and $y$ do not commute, and  $x\incomp y$.
\end{enumerate}
\end{lemma}

\begin{proof}\
By Lemma~\ref{lem:IdRL:properties}, elements in the positive cone commute with each other and the same is true of elements in the negative cone. Hence, if $xy \neq yx$, then $x$ and $y$ must have different signs.

We now consider any $x\in A$. Suppose first that $x$ is central. Then for all $y\in A$, either $xy = x$ and then $x\sle y$, or $xy = y$ and then $yx=xy=y$, that is, $y\sle x$. Also, if $x\comp y$, then $x = xy = yx = y$. So (1) holds. Suppose now that $x$ is not central. Then  there exists $y \in A$ such that $xy\neq yx$. Moreover, since $\m{A}$ is conservative,  $xy=x$ or $xy = y$. We consider these two cases separately:

\begin{itemize}
\item  $xy = x$. Then $yx = y$, so $x\sle y$ and $y\sle x$, i.e., $x\comp y$. If $z \in A$ is another element that does not commute with $x$, then again $x\comp z$ and, by the transitivity of $\comp$, also $y\comp z$. Since $y$ and $z$  both have different signs to $x$, they have the same sign and  $y=z$. So (2) holds. Moreover, in this case, there is no $z\in A$ such that $x\incomp z$. If this were the case, then $y$ and $z$ would have the same sign and either $z\le y$ and $z = xz \le xy = x = zx \le yx = y$, or $y \le z$ and $y = yx \le zx = x = xy \le xz = z$,  contradicting the fact that $x$ has a different sign to both $z$ and $y$. 

\item $xy = y$. Then also $yx= x$, i.e., $x\incomp y$. If $z \in A$ is another element such that $x\incomp z$, then $xz = z$ and $zx = x$. Notice that
\[
y = xy = yxy = yzxy = yxzxy = xzxy = xzy = zy.
\]
Since $y$ and $z$ have the same sign, they commute, and hence $y = yz$, that is, $y\sle z$. By a symmetrical argument, $y\comp z$, and since $y$ and $z$ have the same sign, $y=z$. So (3) holds. \qedhere
\end{itemize}
\end{proof}

Following this last lemma, let us fix for any element $a$ of an idempotent residuated chain $\m A$, the element $a^\sharp$ to be $a$, if $a$ is central, and otherwise, the only element of $\m A$ that does not commute with $a$. Notice that in both cases $(a^\sharp)^\sharp = a$. If $a$ is not central, we call $\{a,a^\sharp\}$ a \emph{noncommuting pair}.

\begin{lemma}\label{lem:blobs}
For any idempotent residuated chain $\m A$, $a\in A$, and $x\in A\setminus\{a,a^\sharp\}$, 
\[
a\sle x \iff a^\sharp \sle x \quad\text{and}\quad
x\sle a \iff x\sle a^\sharp.
\]
\end{lemma}

\begin{proof}
We will prove only the first part, since the second is analogous. We distinguish two cases: $a\comp a^\sharp$ and $a\incomp a^\sharp$. Suppose first that $a\comp a^\sharp$. If $a\sle x$, then $a^\sharp \sle a\sle x$, so $a^\sharp \sle x$. Since $(a^\sharp)^\sharp = a$, this also proves that $a^\sharp\sle x$ implies $a\sle x$. Suppose now that $a\incomp a^\sharp$. Then $x\in A\setminus\{a,a^\sharp\}$ commutes with both $a$ and $a^\sharp$, and is hence $\sle$-comparable with them. If $a\sle x$ and $x\sle a^\sharp$, then  $a\sle a^\sharp$, contradicting $a\incomp a^\sharp$. Hence $a\sle x$ if and only if $a^\sharp \sle x$.
\end{proof}

We now identify properties of the monoidal preorder, analogously to the commutative case, and show that in the finite setting these properties provide a complete description of the algebra. First, let us say that a preorder $\sle$ on a set $A$ is \emph{laced} if
\begin{enumerate}[label={\arabic*.}, itemsep=0.5ex]
\item it has a (unique) greatest element $\ut$;
\item each $a\in A$ is either comparable with all the other elements and we fix $a^\sharp = a$, or there is a unique element $a^\sharp$ such that $a\comp a^\sharp$ or $a\incomp a^\sharp$;
\item for all $a\in A$ and $x\in A\setminus\{a,a^\sharp\}$,
\[
a\sle x \iff a^\sharp \sle x\quad \text{and}\quad x\sle a \iff x\sle a^\sharp.
\]
\end{enumerate}

\noindent
Now let $\m{C} = \pair{C,\le}$ be any chain. We say that a laced preorder $\sle$ on $C$ is \emph{compatible} with $\m{C}$ if
\begin{enumerate}[label={\arabic*.}, itemsep=0.5ex]
\item any least element of $\m{C}$ is also the least element of $\sle$;
\item for all $x,y\in C$, if $\ut\le x,y$, then $x\le y \iff y\sle x$;
\item for all $x,y\in C$, if $x, y\le \ut$, then $x\le y \iff x\sle y$;
\item for each $x\in C$, if $x\neq x^\sharp$, then $\ut\le x \iff x^\sharp \le \ut$.
\end{enumerate}

\begin{theorem}\ \label{thm:representationchains}
\begin{enumerate}[font=\upshape, label={(\alph*)}, itemsep=0.5ex]
\item
The monoidal preorder $\sle$ of any idempotent residuated chain $\m{A}$ is laced and compatible with $\pair{A,\le}$. 

\item
For any  chain $\m{C}=\pair{C,\mt,\jn}$ and compatible laced preorder $\sle$ on $\m{C}$, the algebra $\pair{C,\mt,\jn,\cdot,\ut}$ is an idempotent totally ordered monoid, where 
\[
x\cdot y = 
\begin{cases}
x & \text{if } x\sle y,\\
y & \text{otherwise},
\end{cases}
\]
Moreover, if $\m{C}$ is finite, then $\cdot$ has (uniquely determined) residuals $\ld$ and $\rd$ and $\pair{C,\mt,\jn,\cdot,\ld,\rd,\ut}$ is an idempotent residuated chain.
\end{enumerate}
\end{theorem}

\begin{proof}
Part~(a) is an immediate consequence of Lemmas~\ref{lem:IdRC:cases} and~\ref{lem:blobs}. For part~(b), note first that, since $\sle$ is reflexive and has greatest element $\ut$, the product is idempotent and has identity $\ut$. To prove associativity, we consider $x,y,z\in C$ and distinguish the following cases:
\begin{itemize}
\item If $x\sle y$ and $y\sle z$, then $x\sle z$. So $(xy)z = xz = x = xy = x(yz)$.
\item If $x\sle y$ and $y\not\sle z$, then $(xy)z = xz = x(yz)$.
\item If $x\not\sle y$ and $y\sle z$, then $(xy)z = yz = y = xy = x(yz)$.
\item If $x\not\sle y$ and $y\not\sle z$, then there are four subcases: $z\sle y\sle x$, $z \incomp y \sle x$, $z\sle y\incomp x$, $y\incomp x = z$. In the  first three, we obtain $x\not\sle z$ and hence $x(yz) = xz = z = yz = (xy)z$; in the last, $x(yz) = x(yx) = xx = x = yx = (xy)x = (xy)z$.
\end{itemize}

We check now that the product distributes over joins, which, since $\m C$ is a chain, is equivalent to the monotonicity of the product. Suppose that $x,y,z\in C$ and $y\leq z$. We distinguish the following cases:

\begin{itemize}
\item $\ut\leq x,y,z$ or $x,y,z\leq \ut$. Then the product is, respectively, the join or meet and the result follows immediately.

\item $x\leq \ut \leq y \leq z$. Then $z\sle y$ and $y\not\sle z$. If $x\sle z$, then $x\sle y$ and so $xy = x = xz$. Otherwise, $x\not\sle z$, so $xy$ is $x$ or $y$ and hence  $xy \le z = xz$. To prove $yx\leq zx$, we proceed analogously, noting that if $z\not\sle x$, then $yx = x = zx$, and if $z\sle x$, then $yx \le z = zx$.

\item $x,y \leq \ut\leq z$. Then $xy = yx = x\mt y$, so $xy \le x \le xz$ and $yx \le x \le zx$.

\item $y\leq \ut\leq x, z$. Then $xz = zx = x\jn z$, so $xy \le x \le xz$ and $yx \le x \le zx$.

\item $y\leq z\leq \ut\leq x$. Then $y\sle z$ and $z\not\sle y$. If $x\sle y$, then $x\sle z$ and $xy = x = xz$. If $x\not\sle y$, then $xy = y$, which is smaller than $xz$. For the other inequation, if $y\not\sle x$, then $z\not\sle x$ and then $yx = x = zx$. If $y\sle x$, then $yx = y\le zx$.\end{itemize}
In the case that $\m C$ is finite, it has a least element $\bot$ and, by compatibility, this is also the least element of $\pair{C, \sle}$. That is, $x \bot = \bot = \bot x$ for all $x \in C$. Since $\m C$ also satisfies $x(y\jn z) = xy\jn xz$ for all $x,y,z\in C$, it follows immediately that the product is residuated.
\end{proof}

We use this representation theorem to count the number of idempotent residuated chains of size $n\geq 2$ up to isomorphism. 

\begin{theorem}\label{thm:numberchains}
The number of idempotent residuated chains of size $n\geq 2$  is 
\[
\ct I (n) = \sum_{s=0}^{[\frac n2]-1} \sum_{t=0}^{[\frac n2]-1-s} 2^{n-2(1+s+t)}\binom{n-2-s-t}{n-2(1+s+t), s, t}.
\]
\end{theorem}

\begin{proof}
We determine the number of laced preorders that are compatible with a finite chain  $\pair{A,\le}$  of size $n\geq 2$. Notice that, in every noncommuting pair, one of them will be positive and the other negative. Hence, up to isomorphism, which one is which is not relevant. Moreover, if we let $C=\{a\in A : a=a^\sharp\}$, then ${\sle}\rest C^2$ is a compatible total order on the finite chain $\pair{C,{\le}\rest C^2}$. Hence it suffices to determine the relative positions of the noncommuting pairs and the compatible total preorder of the central elements.

Observe first that the greatest and least elements are fixed, so there are at most $m=n-2$ possible positions for the noncommuting pairs. Let $s$ be the number of comparable noncommuting pairs, and $t$ the number of incomparable (noncommuting) pairs. Then $0\le s\le \big[\frac m2\big] = \big[\frac n2\big]-1$ and $0\le t \le \big[\frac m2\big] - s = \big[\frac n2\big] - 1 - s$, and there are $r=n-2s-2t$ central elements. The number of total positions is $r-2 + s + t$. Hence there are $\binom{r-2+s+t}{s}$ possible choices for the positions of the comparable noncommuting pairs and $\binom{r-2 + t}{t}$ possible choices for the positions of the incomparable pairs. The number of compatible total orders is $2^{r-2}$, so the number of compatible laced preorders with $s$ comparable noncommuting pairs and $t$ incomparable pairs is
\[
f(s,t) = 2^{r-2} \binom{r-2+s+t}{s}\binom{r-2+t}{t} = 2^{r-2}\binom{r-2+s+t}{r-2 + s + t,s,t}.
\]
Hence the number of compatible laced preorders is
\begin{align*}
\ct I(n) &= \sum_{s=0}^{[\frac n2]-1} \sum_{t=0}^{[\frac n2]-1-s} f(s,t)\\
& = \sum_{s=0}^{[\frac n2]-1} \sum_{t=0}^{[\frac n2]-1-s}  2^{r-2}\binom{r-2+s+t}{r-2 + s + t,s,t} \\
&= \sum_{s=0}^{[\frac n2]-1} \sum_{t=0}^{[\frac n2]-1-s} 2^{n-2(1+s+t)}\binom{n-2-s-t}{n-2(1+s+t), s, t}.\qedhere
\end{align*}
\end{proof}

\begin{theorem}
The sequence $(\ct I(n) : n\geq 2)$ satisfies the recurrence formula $\ct I(2) = 1$, $\ct I(3) = 2$, $\ct I(n+2) = 2\ct I(n) + 2\ct I(n+1)$, and the number of idempotent residuated chains of size $n\geq 2$ is
\[
\ct I(n) = \frac{\big(1+\sqrt 3\,\big)^n - \big(1-\sqrt 3\,\big)^n}{2\sqrt 3}.
\]
\end{theorem}

\begin{proof}
It is easy to see that, up to isomorphism, there is only one idempotent residuated chain with $2$ elements and  only two with $3$ elements: one in which the identity is the greatest element and one with a strictly positive element. For any compatible laced preorder on a chain $\pair{A,\leq}$ with $n\geq 4$ elements, there are four possible cases:
\begin{itemize}
\item There is only one $\sle$-cocover $a$ of $\ut$, and $\ut < a$.
\item There is only one $\sle$-cocover $a$ of $\ut$, and $a < \ut$.
\item There are two $\sle$-cocovers $a,a^\sharp$ of $\ut$, and $a\comp a^\sharp$.
\item There are two $\sle$-cocovers $a,a^\sharp$ of $\ut$, and $a\incomp a^\sharp$.
\end{itemize}
It is not difficult to prove that, removing the $\sle$-cocovers of $\ut$, we obtain a compatible laced preorder on a chain of $n-1$ elements in the first two cases, or a chain of $n-2$ elements in the last two cases.

Reciprocally, given a compatible laced preorder on a chain of size $n\geq 2$, we can add a $\sle$-cocover of $\ut$, which can be either positive or negative, or two $\sle$-cocovers of $\ut$, which can consist of comparable or incomparable elements. We obtain a compatible laced preorder on a chain with $n+1$ elements or $n+2$ elements, respectively.
\end{proof}

Recall that the variety of semilinear commutative idempotent residuated lattices is locally finite (\cite{Raf07} and Corollary~\ref{cor:locallyfinite}). This property fails, however, if semilinearity is weakened to distributivity or  idempotence is weakened to being square-increasing, square-decreasing, or $n$-potent for $n \ge 3$ (see~\cite{Raf07}). The next result shows that also commutativity plays an essential role.

\begin{proposition}\label{prop:SemIdRLnotlocallyfinite}
The variety of semilinear idempotent residuated lattices is not locally finite.
\end{proposition}

\begin{proof}
It suffices to exhibit an infinite idempotent residuated chain with a finite set of generators. Consider the set $\Z$ of integers with the standard order, and define $x\cdot y = x$ if $|x| \geq |y|$  and $x\cdot y = y$ otherwise. It is easy to see that this determines the unique structure of an idempotent residuated chain. Moreover,  $x\ld x = |x|$ for each $x \in \Z$, and if $ x>0$, then $x\ld 0 = -x-1$. So we have $1 = |-1|$, $-2 = 1\ld 0$, $2 = |-2|$, $-3 = 2\ld 0$, etc. Also, $0 = (-1)/(-1)$. Hence $\{-1\}$ generates the whole algebra.
\end{proof}

The variety of semilinear idempotent residuated lattices does have the finite embeddability property, however, as will be shown in Corollary~\ref{cor:conservative:varieties:FEP}.


\section{Conservative residuated lattices}\label{sec:conservative}

In the previous two sections, we have studied classes of idempotent residuated lattices that are {\em totally ordered}, i.e., classes satisfying the positive universal formula $(\forall x)(\forall y)(x\jn y \eq x \mathrel{\texttt{or}} x\jn y \eq y)$. In this section, we obtain similar results for classes of idempotent residuated lattices that are {\em conservative}, i.e., classes satisfying the positive universal formula $(\forall x)(\forall y)(xy \eq x  \mathrel{\texttt{or}}  xy \eq y)$. That is, we consider the variety $\va{CsRL}$ generated by the class of  conservative residuated lattices, noting that by congruence distributivity and J{\'o}nsson's Lemma, every subdirectly irreducible member of \va{CsRL} is conservative and any subvariety of \va{CsRL} is generated by its conservative members. 

We show first that any variety $\va{V}$ generated by a class of conservative residuated lattices defined relative to $\va{IdRL}$ by  positive  universal formulas in the language $\{\jn,\cdot,\ut\}$ has the \emph{finite embeddability property}: that is, any finite partial subalgebra of a member of $\va{V}$ embeds into a finite member of $\va{V}$.

\begin{theorem}\label{thm:FEP:for:CsRL}
Let $\va{K}$ be a class of conservative residuated lattices defined relative to $\va{IdRL}$ by positive  universal formulas in the language $\{\jn,\cdot,\ut\}$. Then the variety $\va V$ generated by $\va K$ has the finite embeddability property.
\end{theorem}
\begin{proof}
It suffices to check that the finite embeddability property holds for the subdirectly irreducible members of $\va{V}$. Hence, since $\va{K}$ is a positive universal class, it suffices to show that any finite partial subalgebra $\m{B}$ of some $\m{A}\in\va{K}$ embeds into  a finite member of $\va{K}$. Without loss of generality, we may also assume that $\ut\in B$. First, we let
\[
\top = \sup B, \quad B' = B \cup \{\top\ld \ut, \ut\rd \top\}, \enspace \text{and }\enspace C = B' \cup \{a \mt \ut : a \in B'\}. 
\]
Notice that $C$ is finite, every element $a\in C$ is lower bounded by $a\mt\ut\in C \cap \da\ut$ and, by Lemma~\ref{lem:cones:of:CsRLs:are:chains}, $\pair{C\cap \da\ut,\le}$ is a chain. Hence $C$ has a least element $\bot$.

Now consider the join-subsemilattice $\pair{C^*,\jn}$ of $\pair{A,\jn}$ generated by $C$. The order induced by $\jn$ in $C^*$ is the restriction of the order $\leq$ of $\m A$. Since $C$ is finite and has a least element, $\pair{C^*,\leq}$ is a lattice with the same least element $\bot$. Moreover, by the conservativity of $\m A$, the set $C^*$ is also the universe of a submonoid of $\pair{A,\cdot,\ut}$. That is, $\pair{C^*,\jn,\cdot,\ut}$ is a subalgebra of $\pair{A,\jn,\cdot,\ut}$ and hence every positive universal formula in the language $\{\jn,\cdot,\ut\}$ that is valid in $\m A$ is also valid in $\pair{C^*,\jn,\cdot,\ut}$, in particular, the equations $x(y \jn z) \eq xy \jn xz$ and $(x \jn y)z \eq xz \jn yz$. 

Observe next that $\ut\leq\top$, since $\ut\in B$, and hence that $\top\ld \ut$ and $\ut\rd\top$ are both negative. So $\top = \sup B = \sup B' = \sup C = \sup C^*$. If $\top = \ut$, then $\bot\cdot\top = \bot = \top\cdot\bot$. Otherwise, $\ut < \top$ and $\bot\cdot \top \leq (\ut/\top)\top \leq \ut < \top$ and hence, by conservativity, $\bot\cdot \top = \bot$; similarly, $\top\cdot\bot = \bot$. Hence, in both cases, $\bot \cdot a = \bot = a\cdot\bot$ for all $a \in C^*$. Since $\pair{C^*,\jn,\cdot,\ut}$ also satisfies $x(y \jn z) \eq xy \jn xz$ and is finite, it follows that the product of $C^*$ has residuals given by $a\lld b = \max\{c\in C^* : ac \leq b\}$ and $a\rrd b = \max\{c\in C^* : ca \leq b\}$. It is then straightforward to see that if $a\ld b\in B$, then $a\lld b = a\ld b$, and if $a\rd b\in B$, then $a\rrd b = a\rd b$. Hence the partial algebra $\m B$ embeds into $\pair{C^*,\mt^{\m C^*},\jn,\cdot,\lld,\rrd,\ut}\in\va K$.
\end{proof}

\begin{corollary}\label{cor:conservative:varieties:FEP}
$\va{CsRL}$, $\va{CCsRL}$, and $\va{SemIdRL}$ have the finite embeddability property.
\end{corollary}
\begin{proof}
The result follows directly from Theorem~\ref{thm:FEP:for:CsRL}, since $\va{CsRL}$, $\va{CCsRL}$, and $\va{SemIdRL}$ are generated by the classes of all conservative residuated lattices, commutative conservative residuated lattices, and conservative residuated lattices satisfying $(\forall x)(\forall y)(x\jn y \eq x \mathrel{\texttt{or}} x\jn y \eq y)$, respectively.
\end{proof}

Let us turn our attention now to the variety \va{CCsRL} generated by the class of commutative conservative residuated lattices. As has just been shown, this variety has the finite embeddability property and is therefore generated by its finite members. Also, any finite commutative conservative residuated lattice is subdirectly irreducible, since its negative cone is a finite chain of (central) idempotents. The class $\Si(\va{CCsRL})_\fin$ of finite subdirectly irreducible members of \va{CCsRL} therefore consists of all finite commutative conservative residuated lattices and generates $\va{CCsRL}$. Moreover, using Mace4~\cite{Prov9}, it can be shown that there are  $1$, $2$, $5$, $14$, $42$, $132$, $429$, $1430$, $4862$, and $16796$ such algebras of size $2$ to $11$, respectively. As we show below, it is no accident that this sequence corresponds exactly to the Catalan numbers $C_n=\frac1{n+1}\binom{2n}{n}$.

We first  define a binary operation on $\Si(\va{CCsRL})_\fin$, which we call the \emph{Catalan sum}, and show that the resulting algebra is uniquely determined up to isomorphism by its two summands. For $\m A,\m B\in \Si(\mathsf{CCsRL})_\fin$, we define an algebra $\m C=\m A\cat\m B$ as follows. We let $C$ be the disjoint union of $A$ and $B$ and define a lattice order
\[
{\le^\m C}={\le^\m A}\cup{\le^\m B}\cup(\{\bot^\m A\}\times B)\cup(A\times\ua \ut^\m B).
\]
To specify $\cdot^\m C$, it suffices to define the following monoidal order:
\[
{\sle^\m C}={\sle^\m A}\cup{\sle^\m B}\cup(\{\bot^\m A\}\times B)\cup(B\times (A\setminus\{\bot^\m A\})).
\]
Informally, the monoidal order of $\m C$ is the ordinal sum
$
\{\bot^\m A\}\oplus \pair{B,\sle}\oplus \pair{A\setminus\{\bot^\m A\},\sle}
$.
Since the greatest element is always the identity, it follows that if $\m A$ is nontrivial, then $\ut^\m C=\ut^\m A$ and otherwise $\ut^\m C=\ut^\m B$. The lattice order implies that $\bot^\m A$ is the least element of $\m C$, and that $a\jn b=\ut^\m B\jn b$ whenever $a\in A\setminus\{\bot^\m A\}$ and $b\in B$. Note also that if $\m A$ or $\m B$ is a one element algebra, then the underlying lattice of $\m C$ is simply the ordinal sum of the lattices of $\m A$ and $\m B$ (see Figure~\ref{fig:catalan:sum}).

\begin{figure}
\begin{tikzpicture}[baseline=0pt]
\tikzstyle{every node} = [draw, fill=white, circle, inner sep=0pt, minimum size=5pt]
\tikzstyle{n} = [draw=none, rectangle, inner sep=0pt] 
\tikzstyle{i} = [draw, fill=black, circle, inner sep=0pt, minimum size=5pt]

\draw (0,0) node[label=left:${\bot^\m C =\bot^\m A}$]{}..controls(-1,2)..(0,4)node{}..controls(1,2)..(0,0)--(2,2)node[label=right:${\bot^\m B=\bot^\m B}$]{}..controls(1,4)..(2,6)node{}..controls(3,4)..(2,2)(0,4)--(1.5,5)node[label=left:$\ut^\m B$]{};

\draw (-0.5,3) node[i,label=left:${\ut^\m C=\ut^\m A}$]{};
\draw (0,2) node[n]{$A$};
\draw (2,4) node[n]{$B$};

\draw (7,0) node[label=right:$\bot^\m A$]{};
\draw (7,1.75) node[n,label=right:$B$]{};
\draw (7,4.75) node[n,label=right:$A\setminus\{\bot^\m A\}$]{};
\draw (7,0.75) node[label=right:$\bot^\m B$]{} -- (7,3)node[label=right:$\ut^\m B$]{};
\draw (7,3.75) node[]{} -- (7,6)node[label=right:${\ut^\m C=\ut^\m A}$]{};
\end{tikzpicture}
\caption{The Catalan sum $\m C=\m A\cat\m B$, for nontrivial $\m A$.}
\label{fig:catalan:sum}
\end{figure}
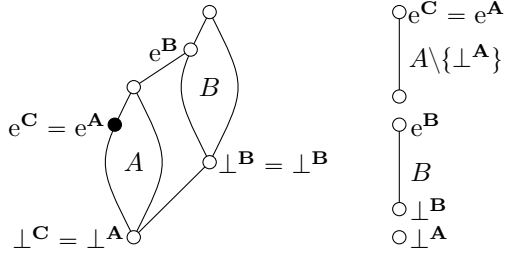

\begin{lemma}\label{lem:Catalan:sum}
If $\m A,\m B\in \Si(\va{CCsRL})_\fin$, then $\m A\cat\m B\in \Si(\va{CCsRL})_\fin$.
\end{lemma}

\begin{proof} 
Let $\m C = \m A\cat\m B$. 
It is straightforward to check that $\pair{C,\le^\m C}$ is a lattice. By assumption, $\m A$ and $\m B$ have total monoidal orders, so by definition the same holds for $\sle^\m C$. Hence $\sle^\m C$ determines a unique commutative monoidal operation. Since $\m A,\m B$ are finite and the least element of $\pair{C,\le^\m C}$ is also the least element of $\pair{C,\sle^\m C}$, to  show that $\m C$ is a residuated lattice, it suffices to check that $x(y\jn z)=xy\jn xz$ for all $x,y,z\in C$. Note that $\{\bot^\m A\}\cup B$ and $A$ are both closed under $\cdot^\m C$ and $\jn^\m C$, hence the identity holds when $x,y,z\in B\cup\{\bot^\m A\}$ or $x,y,z\in A$, and we now check the remaining six cases.

\begin{itemize}
\item $x\in A\setminus\{\bot^\m A\}$, $y,z\in B$: From the definition of $\sle^\m C$, we have $y,z,y\jn^\m B z\sle^\m C x$ and hence $x(y\jn z)=y\jn z=xy\jn xz$.
 
\item $y\in A\setminus\{\bot^\m A\}$, $x,z\in B$: In this case $x\sle^\m C y$ and $y\jn z=\ut^\m B\jn z$, so $x(y\jn z)=x(\ut^\m B\jn z)=x\jn xz=xy\jn xz$. The case $z\in A\setminus\{\bot^\m A\}$, $x,y\in B$ is similar.
 
\item $x,y\in A\setminus\{\bot^\m A\}$, $z\in B$: In this case $y\jn z=\ut^\m B\jn z\sle^C x$ and $xy\in A\setminus\{\bot^\m A\}$, so $x(y\jn z)=y\jn z=\ut^\m B\jn z$ and $xy\jn xz=xy\jn z=\ut^\m B\jn z$. The case $x,z\in A\setminus\{\bot^\m A\}$, $y\in B$ is similar.

\item  $y,z\in A\setminus\{\bot^\m A\}$, $x\in B$: In this case $x\sle^C y,z,y\jn z$, so $x(y\jn z)=x=x\jn x=xy\jn xz$.
\end{itemize}
The algebra $\m C$ is subdirectly irreducible since its negative cone, which is also the negative cone of $\m A$, is a finite chain of central idempotents.
\end{proof}

We now prove a converse to the preceding lemma.

\begin{lemma}\label{lem:Catalan:decomposition}
Suppose that $\m C\in \Si(\va{CCsRL})_\fin$ has size $n\ge 2$. Then $\m C=\m A\cat\m B$ for a pair $\m A,\m B\in \Si(\va{CCsRL})_\fin$ that is unique up to isomorphism.
\end{lemma}

\begin{proof}
Since $\m C$ is finite, commutative, and conservative, the semilattice order  $\sle^\m C$ is a chain. Let $b\in C$ be the unique atom of this chain. The sets $A,B$ are defined by $B=\ua b=\{x\in C:b\le x\}$ and $A=C\setminus B$.  The operations $\mt,\jn,\cdot$ are defined on $A$ and $B$ by restriction from $\m C$. To ensure that these operations are total, we need to prove that both $A$ and $B$ are closed under the operations on $\m C$. This is obvious for the conservative operation $\cdot$. Moreover, since $B$ is a filter, it is the universe of a sublattice. The set $A$ is closed under meets since it is the complement of a filter. Suppose now that $b\le x\jn y$ and $x\neq \bot^\m C \neq y$. If $xy = x$, then $b = xb \le x^2\jn xy = x \jn x = x$; analogously, if $xy = y$, then $b\le y$. If $x = \bot^\m C$, then $b \le x \jn y = y$; and if $y = \bot^\m C$, then $b\le x$. So if $x\jn y\in \ua b$, then $x\in \ua b$ or $y\in \ua b$; that is, $A$ is closed under $\jn$.
 
Let us show that $B$ is an interval of $\langle C,\sle \rangle$. If $b\sle b'' \sle b'$ and $b'\in B$, then $b\le b'$ implies that $b = b'' b \le b'' b' = b''$, and therefore $b''\in B$. Since $C$ is finite, $B = [b,c]_{\sle}$ for some $c\in C$. Hence, for every $x\in B$ we have that $cx = x$. That is, $c$ is the identity on $B$. Notice also that $cb = b\leq c$ implies $b\le c\ld c$, that is, $c\ld c \in B$. So  $c\ld c = c(c\ld c) \le c$. By idempotence, $cc\le c$, and hence $c\le c\ld c$. So $c = c\ld c$. If $\ut^\m C\in B$, then $c = \ut^\m C$ and $A=\{\bot^\m A\}$. Otherwise, $\ut^\m C$ is the identity on $A$.

If $x\in A\setminus\{\bot^\m C\}$, then $cx = c$, yielding $x\le c\ld c = c$. Clearly also $\bot^\m C \le c$. Given $x\in A\setminus\{\bot^\m C\}$ and $y\in B$, if $x\le y$, we have that $c = cx \le cy = y$. So $c$ is the cover of $A$ in $\m C$ if $A\neq\{\bot^\m C\}$. If $A = \{\bot^\m A\}$, then $b$ is the cover of $A$ in $\m C$.

For $y\in B$ and $x\in A\setminus\{\bot^\m C\}$, we have $x\le x\jn y\in B$. Hence $c\le x\jn y$ and  $c\jn y\le x\jn y$. Since also $x\le c$, we have $x\jn y \le c\jn y$. So $x\jn y = c\jn y$.

Finally, if $\bot< x < b$, then $x\in A$, and hence $c = xc\le bc = b \le c$. That is, $\ua b = \{b\}$, which implies $b = \top^\m C$. So $b$ is an atom of $\m C$, unless $b = \top^\m C$.

We have shown that $\m A$ and $\m B$ are commutative conservative  residuated lattices. Moreover, by Lemma~\ref{lem:cones:of:CsRLs:are:chains}, their negative cones are finite chains of central idempotents,  so both are subdirectly irreducible. The proof concludes with the observation that $\m C=\m A\cat\m B$ by definition of the Catalan sum.
\end{proof}

Let us call an algebra \emph{Catalan} if it is a one element algebra (in the language of residuated lattices) or a Catalan sum of Catalan algebras. In particular, if $\m C^1_1$ is a one element algebra, then $\m C^2_1=\m C^1_1\cat\m  C^1_1$ is the two element Boolean algebra. The two three element chains are $\m C^3_1=\m C^1_1\cat \m C^2_1$ and $\m C^3_2=\m C^2_1\cat \m C^1_1$. In general, the algebras of size $n$ are built by constructing all Catalan sums of algebras $\m A$ and $\m B$ of size $n-k$ and $k$ respectively, as $k$ ranges from $1$ to $n-1$ (see Figure~\ref{fig:catalan:sum}).  The following characterization theorem is now an immediate consequence of Lemmas~\ref{lem:Catalan:sum} and~\ref{lem:Catalan:decomposition}

\begin{theorem}\label{thm:Catalan}
The class of finite conservative commutative residuated lattices is precisely the class of Catalan algebras. 
\end{theorem}

This yields the following result.

\begin{theorem}\label{thm:numberconservative}
The number of conservative commutative residuated lattices of $n\ge 1$ elements is $\ct C(n) = \frac 1{n}\binom{2(n-1)}{n-1}$, that is, the $(n-1)$th Catalan number.
\end{theorem}

\begin{proof}
We will prove the result by induction. The sequence $(C_i : i\geq 0)$ of Catalan numbers is determined by $C_0 = 1$ and $C_{n+1} = \sum_{i=0}^n C_iC_{n-i}$. Obviously, $\ct C(1) = 1 = C_0$. Suppose now that $n>1$. Using Lemmas~\ref{lem:Catalan:sum} and~\ref{lem:Catalan:decomposition} and the induction hypothesis,
\begin{align*}
\ct C(n+1) & = \sum_{k=1}^{n} \ct C(k)\cdot \ct C(n+1-k)\\
& = \sum_{k=1}^{n} C_{k-1}C_{n-k}\\
& = \sum_{i=0}^{n-1}C_iC_{n-1-i} = C_{n-1}. \qedhere
\end{align*}
\end{proof}


\section{The amalgamation property}\label{sec:amalgamation}

For the purposes of this paper, a \emph{span} of a class of algebras $\va K$ is a pair of embeddings $\pair{i_1\colon\m A\emb\m B, i_2\colon\m A\emb \m C}$ between algebras $\m A,\m B,\m C\in\va K$. The class $\va K$ is said to have the \emph{amalgamation property} if for every span of $\va K$, there exist an {\em amalgam} $\m D\in \va K$ and embeddings $j_1\colon\m B\emb \m D$ and $j_2\colon\m C\emb \m D$ such that the following diagram commutes:
\[
\begin{tikzcd}[nodes={draw=none}, cells={nodes = {inner sep=2pt}}, column sep = small, row sep = small]
 & \m B\ar[dr, hook, dashed, "j_1"] & \\
\m A \ar[ur, hook, "i_1"]\ar[dr, hook, "i_2" swap] & & \m D\\
 & \m C\ar[ur, hook, dashed, "j_2" swap] &
\end{tikzcd}
\]
In this section, we prove that both the variety of semilinear commutative idempotent residuated lattices and a  noncommutative variety of idempotent residuated lattices have the amalgamation property. Our main tools will be the characterization of commutative idempotent residuated chains provided by Theorem~\ref{thm:representation} and the following criterion for amalgamation in varieties of semilinear residuated lattices obtained in~\cite{MMT14}.

\begin{theorem}[\cite{MMT14}]\label{thm:AP:restricted:SI:2}
Let $\va V$ be a variety of semilinear residuated lattices with the congruence extension property, and let $\cls T$ be the class of finitely generated totally ordered members of $\va V$. If every span in $\cls T$ has an amalgam in $\va V$, then $\va V$ has the amalgamation property.
\end{theorem}

Let us begin by considering the special case of odd Sugihara monoids.

\begin{lemma}[\cite{Dun70}]\label{lem:construction:Sugihara:monoid}
Given any chain $\pair{C,\le}$ and order-preserving involution $\ln$ on $C$ with fixpoint $\ut\in C$, define for $x,y \in C$,
\[
x\cdot y := 
\begin{cases}
x \mt y & \text{if } x \le \ln y\\
x \jn y & \text{if } x > \ln y\\
\end{cases}
\qquad\text{and}\qquad
x\to y := 
\begin{cases}
\ln x\jn y & \text{if } x \le y\\
\ln x \mt y & \text{if } y < x.\\
\end{cases}
\]
Then $\pair{C,\mt,\jn,\cdot,\to,\ut}$ is a totally ordered odd Sugihara monoid.
\end{lemma}

\begin{lemma}\label{lem:odd:Sugihara:amalgamation}
The class of totally ordered odd Sugihara monoids has the amalgamation property.
\end{lemma}

\begin{proof}
Consider a span $\pair{i_1\colon\m A\emb\m B, i_2\colon\m A\emb\m C}$ of totally ordered odd Sugihara monoids, assuming without loss of generality that $i_1$ and $i_2$ are inclusion maps and that $B \cap C = A$. Let $A^-$, $B^-$, and $C^-$ denote the negative cones of $\m{A}$, $\m{B}$, and $\m{C}$  with induced total orders $\le_A$, $\le_B$, and $\le_C$, respectively. We fix $S = A^- \cup B^-$ and let $\le_S$ be any total order extending $\le_B$ and $\le_C$. 

Now consider the set $D=B\cup C$ and the map $\ln\colon D\to D$ given by $\ln^{\m B}\cup\ln^{\m C}$, recalling that $\ln x := x \to \ut$. We also define  $x \le_D y$ if and only if (i) $x,y \in S$ and $x \le_S y$, (ii) $x \in S$, $y \in D\setminus S$, or (iii)  $x,y \in D\setminus S$ and $\ln y \le_S \ln x$. It is easy to check that $\le_D$ is a total order on $D$ extending both $\le_B$ and $\le_C$, and that $S = \{x\in D : x\le_D \ut\}$. Moreover, $\ln$ is an order-preserving involution on $D$ with fixpoint $\ut$. Hence by Lemma~\ref{lem:construction:Sugihara:monoid}, we obtain a totally ordered odd Sugihara monoid $\m D=\pair{D,\mt,\jn,\cdot,\to,\ut}$. Since the order of $\m D$ and the map $\ln$ extend the orders and involutions of $\m B$ and $\m C$, it follows that $\m B$ and $\m C$ are subalgebras of $\m D$, and hence that $\m D$ is the required amalgam.
\end{proof}

Since the variety $\va{OSM}$ of odd Sugihara monoids has the congruence extension property, an application of Theorem~\ref{thm:AP:restricted:SI:2} yields the following result, proved using model-theoretic methods in~\cite{MM12} as part of a full classification of varieties of  Sugihara monoids with the amalgamation property.

\begin{corollary}[Cf.~\cite{MM12}]
The variety of odd Sugihara monoids has the amalgamation property.
\end{corollary}

Note that this last result also follows from the fact that $\va{OSM}$ and the variety of relative Stone algebras are categorically equivalent and the latter has the amalgamation property, as shown in~\cite{GR12} and~\cite{Ma77}, respectively. Let us also mention that $\va{OSM}$ is locally finite by  Corollary~\ref{cor:locallyfinite}; moreover, every $n$-generated subdirectly irreducible member of \va{OSM} is a residuated chain with at most $2n + 2$ elements (see~\cite{Dun70}).

We now turn our attention to the variety $\va{SemCIdRL}$ of semilinear commutative idempotent residuated lattices.

\begin{lemma}\label{lem:CId-chains:have:AP}
The class of commutative idempotent residuated chains has the amalgamation property.
\end{lemma}

\begin{proof}
Let $\pair{i_1\colon\m A\emb\m B, i_2\colon\m A\emb\m C}$ be a span of commutative idempotent residuated chains, assuming without loss of generality that $i_1$ and $i_2$ are inclusion maps and that $B \cap C = A$. Then, using Lemma~\ref{lem:Sugihara:skeleton}, we also have inclusions between {their skeletons} $\m A_{\ga_\ut} \emb \m B_{\ga_\ut}$ and $\m A_{\ga_\ut}\emb\m C_{\ga_\ut}$. Since, by  Proposition~\ref{prop:decomposition:IpCchain}, these {skeletons} are totally ordered odd Sugihara monoids,  Lemma~\ref{lem:odd:Sugihara:amalgamation} yields an amalgam $\m S$ for this span that is also a totally ordered odd Sugihara monoid. Moreover, we may assume that $S=B_{\ga_\ut}\cup C_{\ga_\ut}$.

Consider $a \in A_{\ga_\ut}$. Recalling that  $A_a = \{x\in A : \ga_\ut(x) = a\}$, clearly
\[
A_a \subseteq B_a = \{x\in B : \ga_\ut(x) = a\} \enspace \text{and}\enspace A_a\subseteq C_a =  \{x\in C : \ga_\ut(x) = a\}.
\]
 Moreover, the inclusions $\pair{A_a,{\le}^{\m A}\rest A_a}\emb\pair{B_a,{\le}^{\m B}\rest B_a}$ and $\pair{A_a,{\le}^{\m A}\rest A_a}\emb \pair{C_a,{\le}^{\m C}\rest C_a}$ form a span of chains (viewed as algebras) and have as an amalgam a chain $\pair{D_a,\le_a}$ with $D_a = B_a\cup C_a$. Since $a$ is the greatest element of $A_a$, $B_a$, and $C_a$,  it is also the greatest element of $\pair{D_a,\le_a}$. Now, for all  $b\in B_{\ga_\ut}\setminus A_{\ga_\ut}$ and $c\in C_{\ga_\ut}\setminus A_{\ga_\ut}$, let $\pair{D_b,\le_b} = \pair{B_b,{\le}^{\m B}\rest B_b}$ and $\pair{D_c,\le_c} = \pair{C_c,{\le}^{\m C}\rest C_c}$. Then $\mathcal{X} = \{\pair{D_s,\le_s} : s\in S\}$ is a family of (disjoint) chains such that  each $s\in S$ is the greatest element of $D_s$. By Theorem~\ref{thm:representation},  $\m{D} = \m{S} \otimes \mathcal{X}$ is a commutative idempotent residuated chain satisfying $\m S = (\m{S} \otimes \mathcal{X})_{\ga_\ut}$ and $(S \otimes \mathcal{X})_s = D_s$ for each $s\in S$.

To show that $\m D$ is an amalgam of the original span, it suffices to check that $\m B$ and $\m{C}$ are subalgebras of $\m D$. Consider $x,y\in B$ with $x\in B_{b_1}$ and $y\in B_{b_2}$. Then $b_1,b_2\in S$ and $x\in D_{b_1}$, $y\in D_{b_2}$. If $x\le^\m B y$, then $b_1\le b_2$. If $b_1=b_2=b$, then $x,y\in D_b$ and since the order of $\pair{D_b,\le_b}$ extends the order of $B_b$, we also have  $x\le_b y$, and hence $x\le^\m D y$. If $b_1\neq b_2$, then $b_1\le b_2$, and hence $x\le^\m D y$. This shows that the order of $\m D$ extends the order of $\m B$.

Let us prove next that the product of $\m D$  extends the product of $\m B$. If $b_1 = b_2 =b$, then either $\ut <^\m B b$ and $x\cdot^\m B y = x\jn^\m B y = x\jn^\m D y = x\cdot^\m D y$, or $b <^\m B \ut$, which is analogous. If $b_1\neq b_2$, then
\begin{align*}
x\cdot^\m B y =x &\iff b_1\cdot^\m B b_2 = b_1 \iff b_1\cdot^{\m B_{\ga_\ut}} b_2 = b_1 \iff b_1\cdot^\m S b_2 = b_1\\
 &\iff b_1\cdot^\m D b_2 = b_1 \iff x\cdot^\m D y =x.
\end{align*}

For the residuals:
\begin{align*}
x\to^\m B y 
&= \begin{cases}
\ln^{\m B_{\ga_\ut}} b_1 \jn^\m B y &\text{if } x\le^\m B y,\\
\ln^{\m B_{\ga_\ut}} b_1 \mt^\m B y &\text{if } y <^\m B x
\end{cases}\\
& = \begin{cases}
\ln^\m S b_1 \jn^\m D y &\text{if } x\le^\m D y,\\
\ln^\m S b_1 \mt^\m D y &\text{if } y <^\m D x
\end{cases} \\
&= x\to^\m D y.
\end{align*}
The proof that $\m{C}$ is a subalgebra of $\m{D}$ is symmetrical.
\end{proof}

Every variety of commutative residuated lattices has the congruence extension property, so Theorem~\ref{thm:AP:restricted:SI:2} yields the following result.

\begin{theorem}\label{thm:semilinearidempotentamalgamation}
The variety of semilinear commutative  idempotent residuated lattices has the amalgamation property.
\end{theorem}

We now address an open question in the literature  (see,~e.g.,~\cite{MMT14}) by describing a noncommutative variety of (idempotent) residuated lattices that has the amalgamation property. Note first that all idempotent residuated chains with at most $3$ elements are commutative, since $\bot$ and $\ut$ are central elements and noncentral elements come in pairs. Similarly, there are only two nonisomorphic noncommutative idempotent residuated chains of size $4$: the one represented in Figure~\ref{fig:algebra:C4}, which we call $\m C_4$, and its ``opposite'' (i.e., the algebra resulting from swapping the order of the product).

\begin{figure}
\begin{tikzpicture}[baseline=0pt]
\tikzstyle{every node} = [draw, fill=white, circle, inner sep=0pt, minimum size=5pt]
\tikzstyle{n} = [draw=none, rectangle, inner sep=0pt]
\tikzstyle{i} = [draw, fill=black, circle, inner sep=0pt, minimum size=5pt]
\node at (0,0) [label=left:$\bot$]{};
\node at (0,1) [label=left:$c$\;]{};
\node at (0,2) [i,label=left:$\ut$\;]{};
\node at (0,3) [label=left:$c^\sharp$]{};
\end{tikzpicture}
\qquad\qquad\qquad
\begin{tikzpicture}[baseline=0pt]
\tikzstyle{every node} = [draw, fill=white, circle, inner sep=0pt, minimum size=5pt]
\tikzstyle{n} = [draw=none, rectangle, inner sep=0pt] 
\tikzstyle{i} = [draw, fill=black, circle, inner sep=0pt, minimum size=5pt]
\node at (0,0) [label=right:$\bot$]{};
\node at (-1,1) [label=left:$c$\;]{};
\node at (0,1) [n]{$\comp$};
\node at (1,1) [label=right:\;$c^\sharp$]{};
\node at (0,2) [label=right:\;$\ut$]{};
\end{tikzpicture}
\caption{The algebra $\m C_4$.}
\label{fig:algebra:C4}
\end{figure}
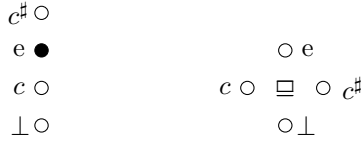

We will use Theorem~\ref{thm:AP:restricted:SI:2} to prove that the variety $\va V(\m C_4)$ generated by $\m C_4$ has the amalgamation property. To show first that $\va V(\m C_4)$ has the congruence extension property, we make use of some results from~\cite{vAl05}. Recall  that a term $u(\vec x,\vec y\,)$ is an \emph{ideal term} in $\vec x$ for a variety $\va V$ with respect to a (term-definable) constant $1$ if and only if $\va V\models t(\vec 1,\vec y\,) \eq 1$. The \emph{ideals} (with respect to $1$) of an algebra $\m A\in\va V$ are the subsets $I\subseteq A$ such that $u(\vec a,\vec b\,)\in I$ for every ideal term $u(\vec x,\vec y\,)$ and $\vec a \in I$,  $\vec b \in A$. If $\m{A}$ is a residuated lattice, then the ideals with respect to $\ut$ coincide with the convex normal subalgebras of $\m{A}$. In~\cite{vAl05} it was proved that the variety generated by a residuated lattice $\m A$ has equationally definable principal congruences (and therefore the congruence extension property) if there exists a finite set $J$ of ideal terms (with respect to $\ut$) such that for all $a,b\in \da\ut$, there exists $u(x,y) \in J$ satisfying
\begin{equation}\label{eq:PIEP*}
b\in \gen{a}^\m A \iff b = u^\m A(a,b),
\end{equation}
where $\gen{a}^\m A$ denotes the convex normal subalgebra generated by $a$.

\begin{lemma}\label{lem:C4:CEP}
$\va V(\m C_4)$ has the congruence extension property.
\end{lemma}

\begin{proof}
Observe first that  $\m C_4$  has only the trivial proper subalgebra, since  $c^\sharp = c\ld \ut = \bot\ld\ut$, $\bot = c^\sharp\ld \ut = \ut\rd c^\sharp$, and $c = \ut\rd c^\sharp$. It suffices now to check that $\m C_4$ satisfies~\eqref{eq:PIEP*} for the set of ideal terms $J = \{\ut,x,(x\ld \ut)\ld \ut, \ut\rd(\ut\rd x)\}$. For $a = \ut$, we have $\gen{\ut}^{\m C_4} = \{\ut\} = J(\ut)$, and for $a \neq \ut$, we have $\gen{a}^{\m C_4} = C_4$ and $J(a) = \{\ut, c, \bot\}$. 
\end{proof}

\begin{theorem}\label{thm:noncommutativeamalgamation}
$\va V(\m C_4)$ has the amalgamation property.
\end{theorem}

\begin{proof}
By Theorem~\ref{thm:AP:restricted:SI:2} and Lemma~\ref{lem:C4:CEP}, it suffices to consider spans of finitely generated totally ordered members of $\va V(\m C_4)$. But $\m C_4$ only has the trivial proper subalgebra $\{e\}$ and, since the lattice of congruences of a residuated lattice is isomorphic to the lattice of its convex normal subalgebras, is simple. By J\'onsson's Lemma, the only nontrivial subdirectly irreducible algebra of $\va V(\m C_4)$ up to isomorphism is $\m C_4$. Hence all spans of finitely generated totally ordered members of $\va V(\m C_4)$ are trivial and clearly have an amalgam. 
\end{proof}

Let us note finally that the same method can  be used to prove that the variety generated by the strongly simple idempotent residuated chain $\m C_{2n+2}$ with $n \ge 2$ pairs of noncommuting elements has the amalgamation property. 



\end{document}